\documentclass[10pt, a4paper, reqno, english]{amsart}

\usepackage{amssymb, mathrsfs}

\usepackage[draft=false,hidelinks]{hyperref}

\theoremstyle{plain}
\newtheorem*{theorem*}{Theorem}
\newtheorem{prop}{Proposition}
\newtheorem{lemma}[prop]{Lemma}
\theoremstyle{remark}
\newtheorem*{remark*}{Remark}

\numberwithin{prop}{section}
\numberwithin{equation}{section}

\newcommand\ab{\mathbf{a}}
\newcommand\xb{\mathbf{x}}
\newcommand\yb{\mathbf{y}}
\newcommand\bb{\mathbf{b}}
\newcommand\gb{\mathbf{g}}
\newcommand\oneb{\mathbf{1}}
\newcommand\ddd{\,\mathrm{d}}
\newcommand\PP{\mathbb{P}}
\newcommand\QQ{\mathbb{Q}}
\newcommand\RR{\mathbb{R}}

\newcommand\ZZ{\mathbb{Z}}
\newcommand\ZZnz{\ZZ_{\ne 0}}
\newcommand\sums[1]{\sum_{\substack{#1}}}
\newcommand\db{{\mathbf{d}}}
\newcommand\eb{{\mathbf{e}}}
\newcommand{\Ps}{\mathscr{P}}
\newcommand{\Gs}{\mathscr{G}}
\newcommand\Munder{\underline{M}}
\newcommand\Mover{\overline{M}}
\newcommand\Aover{\overline{A}}
\DeclareMathOperator{\vol}{vol}
\newcommand{\Mod}[1]{\ (\mathrm{mod}\ #1)}
\newcommand\congr[3]{#1 \equiv #2\Mod{#3}}

\AtBeginDocument{%
   \def\MR#1{}
}

\begin{document}

\title[Points of bounded height on quintic del Pezzo surfaces]
{Points of bounded height on quintic del Pezzo surfaces over the rational numbers}

\author{Christian Bernert}

\address{Institut f\"ur Algebra, Zahlentheorie und Diskrete Mathematik, Leibniz Universit\"at Hannover, Welfengarten 1, 30167 Hannover, Germany}

\email{bernert@math.uni-hannover.de}

\author{Ulrich Derenthal} 

\address{Institut f\"ur Algebra, Zahlentheorie und Diskrete Mathematik, Leibniz Universit\"at Hannover, Welfengarten 1, 30167 Hannover, Germany}

\email{derenthal@math.uni-hannover.de}

\date{May 9, 2025}


\setcounter{tocdepth}{1}

\maketitle

{\centering\footnotesize To Yuri Tschinkel on his 60th birthday.\par}

\begin{abstract}
  We give a relatively short and elementary proof of Manin's conjecture for split smooth quintic del Pezzo surfaces over the rational numbers.
\end{abstract}

\tableofcontents

\section{Introduction}

In his Ph.D. thesis \cite{T92}, Yuri Tschinkel studied rational points of bounded height on del Pezzo surfaces and surfaces of intermediate type. In particular, he considered the split quintic del Pezzo surface over $\QQ$, which is obtained as the blow-up $\pi : X \to \PP^2_\QQ$ of the projective plane in the four points
\begin{equation*}
    p_1 = (1:0:0),\quad p_2=(0:1:0),\quad p_3=(0:0:1),\quad p_4=(1:1:1).
\end{equation*}
It contains ten lines, namely the four exceptional divisors $A_i$ and the six strict transforms $A_{ij}$ of the lines through two of these four points. Since rational points accumulate on these ten lines, one is interested in the rational points in their complement $U$ in $X$.

To define an anticanonical height function, we consider the linear system of cubic polynomials vanishing in the blown-up points $p_1,p_2,p_3,p_4$. A basis is given by each of the sets of six polynomials (for pairwise distinct $i,j,k \in \{1,2,3,4\}$)
\begin{equation*}
    \Ps_1 := \{Y_iY_j(Y_i-Y_k)\}, \qquad \Ps_2 :=\{Y_j(Y_i-Y_j)(Y_i-Y_k)\}.
\end{equation*}
Each of these defines an anticanonical embedding $X \hookrightarrow \PP^5_\QQ$. To obtain the most symmetric height function in Cox coordinates, we will work with their union $\Ps := \Ps_1 \cup \Ps_2$, and hence we define our height function as
\begin{equation}\label{eq:height}
    H(\pi^{-1}(y_1:y_2:y_3)) := \frac{\max_{P \in \Ps}|P(y_1,y_2,y_3)|}{\gcd_{P \in \Ps}P(y_1,y_2,y_3)}
\end{equation}
for any rational point $\xb = \pi^{-1}(y_1:y_2:y_3) \in U(\QQ)$ outside the lines.

\begin{theorem*}
    We have
    \begin{equation*}
        N_{U,H}(B) = c_{X,H} B(\log B)^4 + O(B(\log B)^3\log \log B),
    \end{equation*}
    for $B \ge 3$, where
    \begin{equation*}
        c_{X,H} = \alpha(X) \cdot \omega_\infty(X)\cdot\prod_p \left(1-\frac 1 p\right)^5\left(1+\frac 5 p +\frac {1}{p^2}\right),
    \end{equation*}
    with $\alpha(X)=\frac{1}{144}$ and
    \begin{equation*}
        \omega_\infty(X) = \frac{3}{2}\vol\left\{\yb \in \RR^3 : \max_{P \in \Ps}|P(\yb)| \le 1\right\} = 2\pi^2.
    \end{equation*}
\end{theorem*}

As explained in \cite[\S 1.3]{Bre02} and \cite[\S 2]{BD24}, this agrees with the Manin--Peyre conjecture \cite{FMT89,Pey95}. 

Tschinkel \cite[Theorem~1]{T92} proved the upper bound $N_{U,H}(B) \ll B(\log B)^6$ (see also  \cite[Theorem~1.9]{MT93}).  Salberger improved this to $N_{U,H}(B) \ll B(\log B)^4$ (unpublished, using universal torsors). The above asymptotic formula was first proved by de la Bret\`eche \cite{Bre02} (with a weaker error term, also using universal torsors). Another proof was given by Browning \cite{Br22} (with the same error term as in our result, using conic fibrations). Recently, we generalized these results to arbitrary number fields, also allowing more general height functions \cite{BD24} (with a weaker error term, refining de la Bret\`eche's approach).

Our proof here over $\QQ$ seems to be the shortest and most straightforward one, and will hopefully help the reader who is interested in the case of arbitrary number fields in \cite{BD24}, which follows the same strategy, but is technically much more involved.

\subsection*{Notation and conventions}

We write $(a;b)$ for the greatest common divisor and $[a;b]$ for the lowest common multiple of two integers $a,b$, and similarly for more than two integers. We denote the number of positive divisors of an integer $n$ by $\tau(n)$.

We will use parameters $W=\log B$, $T_1=(\log B)^{2^{22}}$, $T_2=(\log B)^{2^{27}}$; the exact choices of the exponents here are somewhat arbitrary.

If nothing else is stated or clear from the context, $i,j,k,l$ or any subset of them are elements of $\{1,2,3,4\}$ that are pairwise distinct; statements involving them are meant for all such possible choices.

\section{Parameterization and symmetry}

\subsection{Universal torsor parameterization}

The first step in our proof is a parameterization of the rational points of bounded heights by tuples of integers satisfying certain equations, coprimality and height conditions. This can be interpreted as a passage to the universal torsor; see \cite{Skor93}, \cite[\S 1.4]{Bre02}, and the references therein.

\begin{prop}\label{prop:parameterization}
    Let $M(B)$ be the set of $(a_1,a_2,a_3,a_4,a_{12},a_{13},a_{14},a_{23},a_{24},a_{34})$ in $\ZZnz^{10}$ satisfying the height condition
    \begin{equation}\label{eq:height_on_torsor}
        \max|a_{ij}a_ja_{jk}a_ka_{kl}| \le B,
    \end{equation}
    the torsor equation
    \begin{equation}\label{eq:torsor}
        \begin{aligned}
            &a_4a_{14}-a_3a_{13}+a_2a_{12}=0,\\
            &a_4a_{24}-a_3a_{23}+a_1a_{12}=0,\\
            &a_4a_{34}-a_2a_{23}+a_1a_{13}=0,\\
            &a_3a_{34}-a_2a_{24}+a_1a_{14}=0,\\
            &a_{12}a_{34}-a_{13}a_{24}+a_{23}a_{14}=0,
        \end{aligned}  
    \end{equation}
    and the coprimality conditions
    \begin{equation}\label{eq:coprimality}
        (a_i;a_j)=(a_i;a_{jk})=(a_{ij};a_{ik})=1.
    \end{equation}
    Then
    \begin{equation*}
        N_{U,H}(B) = \frac{1}{2^5}|M(B)|.
    \end{equation*}
\end{prop}

\begin{proof}
    Given a point $x \in U(\QQ)$, let $y=(y_1:y_2:y_3) = \pi(x) \in \PP^2(\QQ)$, with coprime integers $y_1,y_2,y_3$. Then we define
    \begin{equation*}
        a_1 = (y_2;y_3),\ a_2 = (y_1;y_3),\ a_3 = (y_1;y_2),\ a_4 = (y_1-y_2;y_1-y_3),
    \end{equation*}
    and
    \begin{equation*}
        a_{12}=\frac{y_3}{a_1a_2}, a_{13}=\frac{y_2}{a_1a_3}, a_{23}=\frac{y_1}{a_2a_3}, a_{14}=\frac{y_2-y_3}{a_1a_4}, a_{24}=\frac{y_1-y_3}{a_1a_4}, a_{34}=\frac{y_1-y_2}{a_3a_4}.
    \end{equation*}
    We observe that $(a_1,\dots,a_{34})$ satisfy \eqref{eq:torsor} and \eqref{eq:coprimality}. They are unique except for the choice of five signs of $a_1,a_2,a_3,a_4$ and $(y_1,y_2,y_3)$ (and hence $(a_{12},\dots,a_{34})$). Furthermore, the monomials appearing in the definition \eqref{eq:height} of $H$ transform into $(a_{ij}a_ia_j)(a_{jk}a_ja_k)(a_{kl}a_ka_l)$, whose greatest common divisor is $a_1a_2a_3a_4$ by \eqref{eq:coprimality}. Hence $H(x) \le B$ is equivalent to \eqref{eq:height_on_torsor}.
\end{proof}

\subsection{Symmetry}

In the process of counting such tuples on the universal torsor, it will be crucial to identify a certain subset of our ten coordinates that are relatively small compared to the others. The following lemma shows that we may assume the $a_i$ to be those distinguished small variables, using the fact that the Weyl group $S_5$ of the root system $A_4$ acts on the set of ten lines \cite[Chapter IV]{Manin} and hence on the tuples in our counting problem. This will allow us to introduce condition \eqref{eq:ai_bound_TB} below; then we will be able to remove the symmetry condition (Proposition~\ref{prop:symmetry_removed}).

\begin{lemma}\label{lem:symmetry}
    Let $\Mover$ be the set of all $(a_1,\dots,a_{34}) \in \ZZnz^{10}$ satisfying \eqref{eq:height_on_torsor}, \eqref{eq:torsor}, \eqref{eq:coprimality}, and the symmetry condition
    \begin{equation}\label{eq:symmetry}
        |a_ia_ja_k| \le |a_{ij}a_{ik}a_{jk}|.
    \end{equation}
    Let $\Munder$ be defined analogously, with $\le$ replaced by $<$ in \eqref{eq:symmetry}. Then
    \begin{equation*}
        5|\Munder(B)| \le |M(B)| \le 5|\Mover(B)|.
    \end{equation*}
\end{lemma}

\begin{proof}
    We consider the involutions $s_l$ on $M(B)$ exchanging $a_{ij}$ with $a_k$, changing the sign of $a_1$ for $l=1$, of $a_{12}$ for $l=2$, of $a_{34}$ for $l=3$, or of $a_4$ for $l=4$, and fixing the remaining three of the coordinates $a_l,a_{il}$. Since $s_l$ maps $|a_1a_2a_3a_4|$ to $|a_{ij}a_{ik}a_{jk}a_l|$, we may assume by symmetry that $|a_1a_2a_3a_4|$ is (one of) the smallest of these expressions (which correspond to pairwise skew lines on $X$) if we introduce the factor $5$, except that $5|\Munder(B)|$ does not count elements with equality, while $5|\Mover(B)|$ counts these elements multiple times.
\end{proof}

In the following, we will derive an asymptotic formula for $|\Mover(B)|$, but all the arguments will be valid for $|\Munder(B)|$ as well, thus eventually leading to an asymptotic formula for $|M(B)|$.

\subsection{Dependent coordinates}

For simplicity, we write $\ab' = (a_1,a_2,a_3,a_4)$ and $\ab'' = (a_{12},a_{13},a_{14},a_{23},a_{24},a_{34})$. After fixing $\ab'$, the remaining six variables in $\ab''$ are restricted to a three-dimensional linear space. One can thus express three of these variables (which we have chosen as $a_{13},a_{24}$ and $a_{14}$) in terms of the remaining three, imposing certain congruence restrictions on the latter, as we describe now.

\begin{lemma} \label{lem:dependent_aij}
    Let $\ab' \in \ZZnz^4$ and $(a_{12},a_{23},a_{34}) \in \ZZ^3$. If
    \begin{align*}
        &\congr{a_3a_{23}}{a_1a_{12}}{a_4},\\
        &\congr{a_4a_{34}}{a_2a_{23}}{a_1}
    \end{align*}
    hold, then
    \begin{equation}\label{eq:dependent_aij}
        \begin{aligned}
            a_{13}&=\frac{a_2a_{23}-a_4a_{34}}{a_1},\\
            a_{24}&=\frac{a_3a_{23}-a_1a_{12}}{a_4},\\
            a_{14}&=\frac{a_2a_3a_{23}-a_3a_4a_{34}-a_1a_2a_{12}}{a_1a_4}
        \end{aligned}
    \end{equation}
    satisfy the torsor equations \eqref{eq:torsor}, with $a_{13},a_{24} \in \ZZ$. If $(a_1;a_4)=1$, then also $a_{14}\in\ZZ$. Otherwise, \eqref{eq:torsor} has no solution $a_{13},a_{24},a_{14} \in \ZZ$.
\end{lemma}

\begin{proof}
    This is a direct computation. The key point is that for $a_1a_2a_3a_4 \ne 0$, any three of the torsor equations in \eqref{eq:torsor} imply the remaining two. Note that the two congruence assumptions immediately imply the integrality of $a_{13}$ and $a_{24}$, while the integrality of $a_{14}$ follows using the coprimality of $a_1$ and $a_4$.
\end{proof}

\section{Restrictions of the counting problem} \label{sec:restrictions}

Before initiating the main counting argument, it will be necessary to perform certain preliminary truncations of the various sets of variables.

\subsection{Restricting the $a_{ij}$}

From the height conditions \eqref{eq:height_on_torsor} and the relations \eqref{eq:torsor}, we can read off typical sizes $|B_{ij}|$ of the variables $a_{ij}$. In general, however, it is not the case that $a_{ij} \ll |B_{ij}|$; in fact, solutions violating this for certain indices constitute a positive proportion of the overall count. This roughly corresponds to the fact that the achimedean density $\omega_{\infty}$ is the volume of an unbounded region. However, we are able to show that solutions moderately high in such "cusps" yield a negligible contribution and can henceforth be ignored. This truncation is crucial to control the error terms in the later stages of the argument, and constitutes one of the novelties in our argument compared to the one in \cite{Bre02}.

We define
\begin{equation}\label{eq:def_Bij}
    B_{ij} = \frac{(B|a_1a_2a_3a_4|)^{1/3}}{a_ia_j}
\end{equation}
and consider the condition
\begin{equation}\label{eq:restriction_aij}
    |a_{ij}| \le W|B_{ij}|
\end{equation}
for all $i,j$ (using \eqref{eq:dependent_aij} for $a_{13},a_{24},a_{14}$) in order to define
\begin{align*}
    A(W,\ab',B) &= \{\ab'' \in \ZZnz^6 : \eqref{eq:height_on_torsor}, \eqref{eq:torsor}, \eqref{eq:coprimality}, \eqref{eq:restriction_aij}\},\\
    \Aover(W,\ab',B) &= \{\ab'' \in \ZZnz^6 : \eqref{eq:height_on_torsor}, \eqref{eq:torsor}, \eqref{eq:coprimality}, \eqref{eq:symmetry}, \eqref{eq:restriction_aij}\}.
\end{align*}
Since all $a_j,a_{kl}$ are nonzero, \eqref{eq:height_on_torsor} implies that these sets are empty unless all
\begin{equation}\label{eq:ai_bound_B}
    |a_i| \le B.
\end{equation} 

\begin{prop} \label{prop:restrict_W}
    For $W \ge 1$, we have
    \begin{equation*}
        |\Mover(B)| = \sum_{\ab' \in \ZZnz^4} |\Aover(W,\ab',B)| + O(W^{-2}B(\log B)^4).
    \end{equation*}
\end{prop}

\begin{proof}
    We need to discard solutions where one of the conditions \eqref{eq:ai_bound_B} is violated. Renormalize the variables to $z_{ij}:=a_{ij}/B_{ij}$ so that four of the torsor equations \eqref{eq:torsor} become $z_{ij} \pm z_{ik} \pm z_{il}=0$ while the height conditions \eqref{eq:height_on_torsor} imply $z_{ij}z_{jk}z_{kl} \ll 1$.

    Let us first suppose that $\max |z_{ij}| \asymp W$ for some $W>1$. The torsor equations then imply that at least two variables in one equation must be $\asymp W$. By symmetry, we may suppose that $|z_{13}|, |z_{34}| \asymp W$. The height conditions now imply $z_{12}, z_{24} \ll \frac{1}{W^2}$ and then the torsor equations also imply $z_{23} \ll \frac{1}{W^2}$.

    To summarize, we have $z_{12}, z_{23} \ll \frac{1}{W^2}$ and $z_{34} \ll W$. This choice corresponds to working on the set $E_2$ in \cite{Bre02} and is beneficial for our choice $(a_{12},a_{23},a_{34})$ of summation variables, exploiting that while one of them can be relatively large, the other two will be much smaller, leading to an eventual saving over the size of the main term.

    However, we still need to take care of the error terms arising from the congruence conditions. This is where the symmetry assumption \eqref{eq:symmetry} enters the picture, ensuring that the moduli $a_i$ in the congruences from Lemma \ref{lem:dependent_aij} are not too big compared to the variables $a_{jk}$ we are summing over.

    Indeed, \eqref{eq:symmetry} now implies that
    \begin{equation*}
    |a_1a_3a_4| \ll |a_{13}a_{14}a_{34}| \ll W^3 |B_{13}B_{14}B_{34}|=W^3 \frac{B|a_1a_2a_3a_4|}{|a_1a_3a_4|^2}
    \end{equation*}
    and hence $|a_1a_3a_4| \ll W (B|a_1a_2a_3a_4|)^{1/3}$. This is equivalent to $a_1 \ll W|B_{34}|$. Similarly, we obtain that $a_4 \ll \frac{1}{W}|B_{23}|$.

    Now the number of such solutions can be estimated by
    \begin{align*}
    &\sum_{a_1,a_2,a_3,a_4} \sum_{a_{12} \ll |B_{12}|/W^2} \sums{a_{23} \ll |B_{23}|/W^2\\a_{23} \Mod{a_4}} \sums{a_{34} \ll W|B_{34}|\\ a_{34} \Mod{a_1}} 1\\
    &\ll \sum_{a_1,a_2,a_3,a_4} \frac{|B_{12}|}{W^2} \cdot \left(\frac{|B_{23}|}{|a_4|W^2}+1\right)\left(\frac{W|B_{34}|}{|a_1|}+1\right)\\
    &\ll \sum_{a_1,a_2,a_3,a_4} \frac{|B_{12}|}{W^2} \cdot \frac{|B_{23}|}{W|a_4|} \cdot \frac{W|B_{34}|}{|a_1|}=\frac{1}{W^2} \sum_{a_1,a_2,a_3,a_4} \frac{|B_{12}B_{23}B_{34}|}{|a_1a_4|}\\
    &=\frac{1}{W^2} \sum_{a_1,a_2,a_3,a_4} \frac{B}{|a_1a_2a_3a_4|} \ll \frac{1}{W^2} B(\log B)^4.
    \end{align*}
    Here, we use Lemma \ref{lem:dependent_aij} in the first step,
    \begin{equation}\label{eq:sum_congruence}
        |\{n \in \ZZ : 0 < n \le t,\ \congr{n}{c}{a}\}| = \frac{t}{|a|} + O(1)
    \end{equation}
    (which holds for any $a \in \ZZnz$ and $c \in \ZZ$) in the second step, the estimations above in the third step, and \eqref{eq:ai_bound_B} in the final step.

    Thus the number of solutions with $\max |z_{ij}| \asymp W$ is $O(W^{-2}B(\log B)^4)$ for all $W>1$. The claim follows by a dyadic summation over the range of $W$.
\end{proof}

\begin{lemma}\label{lem:ai_bound_WB}
    For $W \ge 1$, let $\ab' \in \ZZnz^4$ be such that $\Aover(W,\ab',B)$ is nonempty. Then $\ab'$ satisfies
    \begin{align}
        |a_i^2a_j^2a_k^2a_l^{-1}| &\le W^3B,\label{eq:ai_bound_WB_1}\\
        |a_i| &\le W|B_{jk}|.\label{eq:ai_bound_WB_2}
    \end{align}
\end{lemma}

\begin{proof}
    Choose $\ab'' \in \Aover(W,\ab',B)$. Combining \eqref{eq:restriction_aij} with \eqref{eq:symmetry}, the first inequality holds because of
    \begin{equation*}
        |a_ia_ja_k| \le |a_{ij}a_{ik}a_{kl}| \le W^3|B_{ij}B_{ik}B_{jk}| = \frac{W^3B|a_1a_2a_3a_4|}{|a_i^2a_j^2a_k^2|}.
    \end{equation*}
    The second one is equivalent to the first using definition \eqref{eq:def_Bij}.
\end{proof}

\begin{lemma} \label{lem:bound_a''}
    For $\ab' \in \ZZnz^4$, we have
    \begin{equation*}
        |\Aover(W,\ab',B)| \ll \frac{B}{|a_1a_2a_3a_4|}.
    \end{equation*}
\end{lemma}

\begin{proof}
    In the notation of the proof of Proposition~\ref{prop:restrict_W}, this follows in the same way as long as $\max |z_{ij}|>1$. However, if $\max |z_{ij}| \le 1$, we can still run the same argument with $W=1$, noting that we only used the upper bounds in the second part of that proof. 
\end{proof}

\begin{remark*}
    Note that this already proves the upper bound $O(B(\log B)^4)$ for the total count by summing over $a_1,a_2,a_3,a_4$.
\end{remark*} 

\subsection{Restricting the $a_i$}

Next, we improve \eqref{eq:ai_bound_WB_1} to
\begin{equation}\label{eq:ai_bound_TB}
    |a_i^2a_j^2a_k^2a_l^{-1}| \le T_2^{-1}B,
\end{equation}
which is equivalent to
\begin{equation}\label{eq:ai_bound_TB_2}
    |a_i|\le T_2^{-1/3}|B_{jk}|.
\end{equation}

The improvement from \eqref{eq:ai_bound_WB_2} to \eqref{eq:ai_bound_TB_2} ensures that the error terms arising from the congruence conditions are not just controllable, but in fact negligible compared to the main term, even after summing over the Möbius variables later.

\begin{lemma} \label{lem:bound_a'_B}
    For $T_2 \ge 1$ and $\ab \in \ZZnz^4$ with \eqref{eq:ai_bound_TB}, we have $|a_i| \le B$.
\end{lemma}

\begin{proof}
    The product of the four inequalities \eqref{eq:ai_bound_TB} is $|a_1\dots a_4|^5 \le T_2^{-4}B^4 \le B^4$.
\end{proof}

\begin{lemma}
    For $T_2 = (\log B)^{2^{27}}$ and $W=\log B$, we have
    \begin{equation*}
        |M(B)| = \sums{\ab' \in \ZZnz^4\\\eqref{eq:ai_bound_TB}} |\Aover(W,\ab',B)|+O(B(\log B)^3 \log\log B).
    \end{equation*}
\end{lemma}

\begin{proof}
    We need to discard solutions where \eqref{eq:ai_bound_TB} is violated. By symmetry and Lemma~\ref{lem:ai_bound_WB}, it suffices to bound the number of solutions with
    \begin{equation*} 
    T_2^{-1}B <|a_1^2a_2^2a_3^2a_4^{-1}| \le W^3B.
    \end{equation*}
    But for fixed $a_1,a_2,a_3$, this restricts $a_4$ to a union $I$ of $O(\log\log B)$ dyadic intervals. Since $\sum_{a_4 \in I} \frac{1}{|a_4|} \ll \log\log B$, the bound follows from Lemma~\ref{lem:bound_a''} by summing over $a_1,a_2,a_3$ trivially with the help of Lemma~\ref{lem:bound_a'_B}.
\end{proof}

\subsection{Removing the symmetry condition}

Finally, we remove the condition \eqref{eq:symmetry} appearing in $\Aover(W,\ab',B)$.

\begin{prop}\label{prop:symmetry_removed}
    For $T_2 = (\log B)^{2^{27}}$ and $W=\log B$, we have
    \begin{equation*}
        |M(B)| = \sums{\ab' \in \ZZnz^4\\\eqref{eq:ai_bound_TB}} |A(W,\ab',B)| + O(B(\log B)^3 \log \log B).
    \end{equation*}
\end{prop}

\begin{proof}
    We need to discard solutions where \eqref{eq:symmetry} is violated. By symmetry, we may assume that $|a_{12}a_{13}a_{23}|<|a_1a_2a_3|$. Hence, we have
    \begin{equation*}
       \frac{|a_{12}a_{13}a_{23}|}{|B_{12}B_{23}B_{13}|}<\frac{|a_1a_2a_3|}{|B_{12}B_{23}B_{13}|}=\frac{|a_1^2a_2^2a_3^2|}{|Ba_4}| \le \frac{1}{T_2} 
    \end{equation*}
    by \eqref{eq:ai_bound_TB}. By symmetry, we may suppose that $|a_{12}|<T_2^{-1/3} |B_{12}|$. But now, arguing as in the proof of Proposition~\ref{prop:restrict_W}, we can bound the total number of such solutions by $T_2^{-1/3} B(\log B)^4$, which is satisfactory.
\end{proof}

\section{Estimation of the main contribution}

\subsection{Möbius inversion}

We begin by removing the coprimality conditions involving the variables $a_{ij}$. Via Möbius inversion, we obtain a lattice point counting problem, involving congruences modulo the newly introduced Möbius variables $\db=(d_1,d_2,d_3,d_4),\ \eb=(e_1,e_2,e_3,e_4) \in \ZZ_{>0}^4$.

To describe our first result, consider the conditions
\begin{align}
    &(d_i;a_j)=1,\label{eq:db}\\
    &e_i \mid a_i,\label{eq:eb}\\
    &f_{ij}:=[d_i;d_j]e_ke_l \mid a_{ij}.\label{eq:moebius}
\end{align}
It will also be helpful to use the notation $\mu(\db)=\mu(d_1)\mu(d_2)\mu(d_3)\mu(d_4)$, the analogous $\mu(\eb)$ as well as $\mu(\db,\eb)=\mu(\db)\mu(\eb)$.
Since we keep the coprimality conditions concerning only the $a_i$ untouched for the moment, it will be convenient to write 
\begin{equation*} 
    \theta_0(\ab')=\begin{cases} 1, & (a_i;a_j)=1,\\0, & \text{else}. \end{cases}
\end{equation*}

Consider the set $\Gs=\Gs(\ab',\db,\eb) \subset \ZZ^3$ of all $(a_{12},a_{23},a_{34})$ for which all six coordinates $a_{ij}$ are integral and satisfy \eqref{eq:moebius}. Moreover, let 
\begin{equation*}
    S^{(W)}(\ab',B) = \{(a_{12},a_{23},a_{34}) \in \RR_{\ne 0}^3 : \eqref{eq:height_on_torsor},\ \eqref{eq:restriction_aij},\ a_{ij} \ne 0\}
\end{equation*}
(using \eqref{eq:dependent_aij} for $a_{13},a_{24},a_{14}$ in all three conditions).

\begin{prop}\label{prop:moebius}
    For $\ab' \in \ZZnz^4$ with $\theta_0(\ab')=1$, we have
    \begin{equation*}
        |A(W,\ab',B)|=\sums{\db: \eqref{eq:db}\\\eb: \eqref{eq:eb}} \mu(\db,\eb) |\Gs(\ab',\db,\eb) \cap S^{(W)}(\ab',B)|.
    \end{equation*}
\end{prop}

\begin{proof}
   The conditions $(a_i;a_j)=1$ from \eqref{eq:coprimality} are encoded in $\theta_0(\ab')$. We use M\"obius inversion to remove the other coprimality conditions, starting with $(a_{ij};a_{ik})=1$. Any divisor of $d_i$ of $a_{ij}$ and $a_{ik}$ is coprime to $a_l$ (since $(a_{ij};a_l)=1$), hence $d_i$ must also divide $a_{il}$ by the $i$-th torsor equation \eqref{eq:torsor}. Furthermore, $d_i \mid a_{ij}$ implies $(d_i;a_j)|(a_{ij};a_j)=1$, hence we may restrict the M\"obius inversion to \eqref{eq:db}. Writing $A=A(W,\ab',B)$, this proves
    \begin{equation*}
        |A| = \sum_{\db:\eqref{eq:db}} \mu(\db) |\{\ab'' \in \ZZnz^6 : \eqref{eq:height_on_torsor}, \eqref{eq:torsor}, \eqref{eq:restriction_aij},\ (a_i;a_{jk})=1,\ d_i \mid a_{ij},a_{ik},a_{il}\}.
    \end{equation*}

    Another M\"obius inversion removes the conditions $(a_i;a_{jk})=1$. A divisor $e_i$ of $a_i$ and $a_{jk}$ is coprime to $a_l$ (since $(a_i;a_l)=1$), hence $e_i$ must also divide $a_{jl}$ by the $j$-th torsor equation; analogously, $e_i$ must divide $a_{kl}$. Hence
    \begin{equation*}
        |A| = \sums{\db:\eqref{eq:db}\\\eb:\eqref{eq:eb}} \mu(\db,\eb) |\{\ab'' \in \ZZnz^6 : \eqref{eq:height_on_torsor}, \eqref{eq:torsor}, \eqref{eq:restriction_aij},\ d_i \mid a_{ij},a_{ik},a_{il},\ e_i \mid a_{jk},a_{jl},a_{kl}\}.
    \end{equation*}
    Finally, $a_{ij}$ being divisible by $d_i,d_j,e_k,e_l$ is equivalent to \eqref{eq:moebius} since $e_k\mid a_k$ and $e_l \mid a_l$ while $(d_id_j;a_ka_l)=(a_k;a_l)=1$.
\end{proof}

The next step is to describe the set $\Gs$ in terms of congruence conditions on $a_{12}, a_{23}$, and $a_{34}$. To this end, let
\begin{align*}
    b_{12}&:=e_3e_4[d_1;d_2;(d_3;d_4)],\\
    b_{23}&:=e_1[d_2;d_3;d_4],\\
    b_{34}&:=e_2[d_1;d_3;d_4].
\end{align*}

\begin{lemma}\label{lem:lattice}
    Fix $\ab' \in \ZZnz^4$ with $\theta_0(\ab')=1$, and $\db,\eb \in \ZZ_{>0}^4$ with \eqref{eq:db} and \eqref{eq:eb}.
    
    The set $\Gs=\Gs(\ab',\db,\eb)$ is a complete lattice, described by congruence conditions
    \begin{equation*}
        \congr{a_{12}}{0}{b_{12}},\quad \congr{a_{23}}{\gamma_{23}}{a_4b_{23}},\quad \congr{a_{34}}{\gamma_{34}}{a_1b_{34}}
    \end{equation*}
    for some $\gamma_{23}$ only depending on $\ab', \db,\eb$ and $a_{12}$, and $\gamma_{34}$ only depending on $\ab',\db,\eb$ and $a_{12},a_{23}$.
\end{lemma}

\begin{proof}
    It is clear by linearity of all given conditions that $\Gs$ is a lattice. Our job is to transform the six conditions from \eqref{eq:moebius} in a form only depending on $a_{12},a_{23},a_{34}$.

    In the first step, we can combine \eqref{eq:moebius} for the dependent coordinates $a_{ij}$ with Lemma~\ref{lem:dependent_aij} to describe $\Gs$ as the set of $(a_{12},a_{23},a_{34}) \in \ZZ^3$ satisfying $f_{12} \mid a_{12}, f_{23} \mid a_{23}, f_{34} \mid a_{34}$ as well as
    \begin{equation*}
     a_4f_{24} \mid a_3a_{23}-a_1a_{12}, \quad a_1f_{13} \mid a_1a_{34}-a_2a_{23}, 
    \end{equation*}
    the latter being equivalent to
    \begin{equation*}
        a_4[d_2;d_4] \mid a_3a_{23}-a_1a_{12}, \quad a_1[d_1;d_3] \mid a_1a_{34}-a_2a_{23},
    \end{equation*}
    since the divisibility by the respective $e_i$ is automatic and they are coprime to the remaining part of the modulus. Note that now the moduli are coprime to the coefficients $a_3$ and $a_1$, respectively, so that these are really congruence conditions for $a_{23}$ and $a_{34}$, respectively.
    
    Here, we note that the final condition $a_1a_4f_{14} \mid a_2a_3a_{23}-a_3a_4a_{34}-a_1a_2a_{12}$ is automatically satisfied under these circumstances, as can be easily checked.

    The next step is to combine these conditions and check their compatibility. We have two congruences for $a_{34}$, one modulo $f_{34}=[d_3;d_4]e_1e_2$ and one modulo $a_1[d_1;d_3]$.
    The greatest common divisor of the two moduli is $e_1 \cdot [d_3;(d_1;d_4)]$. So the two conditions are compatible if and only if
    \begin{equation*}
        e_1 \cdot [d_3;(d_1;d_4)] \mid a_2a_{23},
    \end{equation*}
    which is clearly equivalent to $[d_3;(d_1;d_4)] \mid a_{23}$.
    In that case, the coordinate $a_{34}$ is determined by a congruence condition modulo the least common multiple of the two moduli, which is precisely $b_{34}$.

    Next, we can rewrite the conditions for $a_{23}$ as
    \begin{align*}
        [d_2;d_3;(d_1;d_4)]e_1e_4 &\mid a_{23},\\
        a_4[d_2;d_4] &\mid a_3a_{23}-a_1a_{12}.
    \end{align*}
    The greatest common divisor of the two moduli is $e_4 [d_2;(d_1;d_4);(d_3;d_4)]$. So the two conditions are compatible if and only if
    \begin{equation*}
        e_4 [d_2;(d_1;d_4);(d_3;d_4)] \mid a_1a_{12},
    \end{equation*}
    which is clearly equivalent to $(d_3;d_4) \mid a_{12}$.
    In that case, the coordinate $a_{23}$ is determined by a congruence condition modulo the least common multiple of the two moduli, which is precisely $b_{23}$.

    Finally, we can combine the two divisibilities for $a_{12}$ to one by the least common multiple of the two moduli, which is precisely $b_{12}$.
    
    The lattice described by these congruence conditions is clearly complete.
\end{proof}

\subsection{Large Möbius variables} \label{sec:large_moebius}

We now restrict the variables $\db,\eb$ to moderate sizes, using a lifting argument as in \cite[\S 5]{Bre02}. This allows us to sum the error terms in the later steps of the argument.

The first ingredient here is an upper bound on the number of solutions without the coprimality conditions. Since we want to improve the error term from \cite{Bre02}, we need a version of this also involving a sum over divisor functions.

\begin{lemma}\label{lem:upperbound}
    We have
    \begin{equation*}
        \sums{(\ab',\ab'') \in \ZZnz^{10}\\\eqref{eq:height_on_torsor},\eqref{eq:torsor}} (\tau(a_1)\tau(a_2)\tau(a_3)\tau(a_4))^5 \ll B(\log B)^{2^{21}}.
    \end{equation*}
\end{lemma}

\begin{proof}
    By Hölder's inequality and symmetry, it suffices to prove that
    \begin{equation*}
        \sums{(\ab',\ab'') \in \ZZnz^{10}\\\eqref{eq:height_on_torsor},\eqref{eq:torsor}} \tau(a_2)^{20} \ll B(\log B)^{2^{21}}.
    \end{equation*}
    Moreover, by the symmetry argument from Lemma \ref{lem:symmetry}, it suffices to prove both this and the same bound with $\tau(a_{12})^{20}$ under the symmetry assumption \eqref{eq:symmetry}.

    Finally, by the argument from Proposition \ref{prop:restrict_W}, it suffices to prove a bound of $B(\log B)^{2^{21}}/W$ in either of the two situations, for either of the two cases
    \begin{equation} \label{eq:case_e2}
        z_{12},z_{23},z_{24} \ll \frac{1}{W^2},\quad z_{13},z_{14},z_{34} \ll W
    \end{equation}
    and
    \begin{equation} \label{eq:case_e3}
        z_{13},z_{23},z_{34} \ll \frac{1}{W^2},\quad z_{12},z_{14},z_{24} \ll W
    \end{equation}
    with some $W \ge 1$ (note that the choice $W=1$ takes care of the case where \eqref{eq:restriction_aij} holds).

    Let us first consider the situation where we are summing $\tau(a_2)^{20}$. If \eqref{eq:case_e2} holds, then by the same line of argument as in the proof of Proposition~\ref{prop:restrict_W}, we can bound the expression by
    \begin{equation*} 
    \sum_{\ab'} \frac{B}{W^2|a_1a_2a_3a_4|} \cdot (a_3;a_4) \cdot (a_1;a_4) \cdot \tau(a_2)^{20},
    \end{equation*}
    the only difference being that in general $a_{23}$ is only determined modulo $\frac{a_4}{(a_3;a_4)}$ and $a_{34}$ is determined modulo $\frac{a_1}{(a_1;a_4)}$.
    
    If \eqref{eq:case_e3} holds, we instead obtain the bounds $a_{12} \ll |B_{12}|W$, $a_{23} \ll |B_{23}|/W^2$, $a_{34} \ll |B_{34}|/W^2$ together with $a_1 \ll |B_{34}|/W$ and $a_4 \ll |B_{23}|/W$ so that a similar computation yields the slightly weaker bound
    \begin{equation*} 
    \sum_{\ab'} \frac{B}{W|a_1a_2a_3a_4|} \cdot (a_3;a_4) \cdot (a_1;a_4) \cdot \tau(a_2)^{20},
    \end{equation*}
    which therefore holds in both cases.
    
    Writing $d=(a_1;a_4)$, $e=(a_3;a_4)$ and $a_1'=a_1/d$, $a_3'=a_3/e$, the expression can now be bounded as
    \begin{align*}
        &\ll\frac{B}{W} \sum_{a_2=1}^B \frac{\tau(a_2)^{20}}{a_2} \cdot \sum_{a_4=1}^B \sums{d \mid a_4\\e \mid a_4} \sum_{a_1'=1}^B \sum_{a_3'=1}^B\frac{1}{a_1'a_3'a_4}\\
    &\ll \frac{B(\log B)^2}{W} \left(\sum_{a_2=1}^B \frac{\tau(a_2)^{20}}{a_2}\right) \sum_{a_4=1}^B \frac{\tau(a_4)^2}{a_4}\\
    &\ll \frac{B(\log B)^{2+2^2+2^{20}}}{W},
    \end{align*}
    using familiar estimates for moments of the divisor function \cite[(2.31)]{MV}.

    On the other hand, when we are summing $\tau(a_{12})^{20}$ instead, the sum over $a_{12}$ inside introduces an additional factor of $(\log B)^{2^{20}-1}$, whereas the sum over $a_2$ saves the same factor, thus leading to the same bound.
\end{proof}

\begin{prop}\label{prop:restrict_moebius}
    We have
    \begin{align*}
        |\Mover(B)|={}&\sums{\ab' \in \ZZnz^4\\\eqref{eq:ai_bound_TB}} \theta_0(\ab') \sums{\db: \eqref{eq:db},\ |d_i| \le T_1\\\eb: \eqref{eq:eb},\ |e_i| \le T_1} 
        \mu(\db,\eb)|\Gs(\ab',\db,\eb) \cap S^{(W)}(\ab',B)|\\
        &+O(B(\log B)^3 \log\log B).
    \end{align*}
\end{prop}

\begin{proof}
    Combining Proposition~\ref{prop:moebius} with Proposition~\ref{prop:symmetry_removed}, we need to discard tuples $(\ab',\ab'',\db,\eb)$ with $\max d_i>T_1$ or $\max e_i>T_1$. The idea is that such a large common divisor should allow us to construct a solution of significantly smaller height, the number of which we can bound by Lemma \ref{lem:upperbound}. This construction consists of two steps and is essentially reversible, the number of choices being counted by divisor functions, of which we carefully keep track in order to obtain a good error term. It would be more straightforward if we allowed ourselves to bound the divisor function pointwise, but this would result in a worse error term, see \cite[Proposition~5.4]{BD24}.

    The first step is to replace the $d_i$ by numbers $g_i$ which are pairwise coprime. Indeed, for each prime we can choose one of the $d_i$ divisible by it to a maximal power. We can then define $g_i$ to be the product of all these prime powers associated with $d_i$. In this way, we obtain pairwise coprime numbers $g_i \mid d_i$. Moreover, note that $g_1g_2g_3g_4 \ge \max d_i$. Finally, note that from given $g_i$, we can reconstruct the $d_i$ as divisors of $g_1g_2g_3g_4$ up to $\tau(g_1g_2g_3g_4)^4$ many choices.

    In the second step, we use the $e_i$ together with the newly created $g_i$ to obtain from a given solution $(\ab',\ab'')$ a new solution $(\bb',\bb'')$ with $b_i=\frac{a_ig_i}{e_i}$ and $b_{ij}=\frac{a_{ij}}{g_ig_je_ke_l}$. This remains integral by construction (using the coprimality of the $g_i$), and has height
    \begin{equation*}
        \ll \frac{B}{e_1^2e_2^2e_3^2e_4^2g_1g_2g_3g_4}.
    \end{equation*}
    
    Now first suppose that $e_1>T_1$. Since we can reconstruct $(\ab',\ab'')$ uniquely from $(\bb',\bb'',\gb,\eb)$, the total number of solutions can be bounded by
    \begin{equation*} \sum_{\gb,\eb: e_1>T_1} \sum_{\bb,\bb'} (\tau(g_1g_2g_3g_4))^4 \ll \sum_{\gb,\eb: e_1>T_1} \frac{B(\log B)^{2^{21}}}{g_1g_2g_3g_4e_1^2e_2^2e_3^2e_4^2} \ll \frac{B(\log B)^{2^{22}}}{T_1}
    \end{equation*}
    using Lemma \ref{lem:upperbound}, which is satisfactory.

    Next suppose that $\max d_i>T_1$, so that $g_1g_2g_3g_4>T_1$. We then note that from $(\bb',\bb',\eb)$ we can reconstruct the $g_i$ as divisors of $b_ie_i$ up to $\tau(b_ie_i)$ many choices. Therefore, the total number of solutions can be bounded as
    \begin{align*}
    &\sum_{\eb} \sum_{\bb',\bb''} (\tau(b_1e_1)\tau(b_2e_2)\tau(b_3e_3)\tau(b_4e_4))^5\\
    &\ll \sum_{\eb} \prod_i \tau(e_i)^5 \sum_{\bb',\bb''} (\tau(b_1)\tau(b_2)\tau(b_3)\tau(b_4))^5.
    \end{align*}
By the previous lemma and the fact that $g_1g_2g_3g_4>T_1$, this can be bounded as
    \begin{equation*}
        \ll \sum_{\eb} \prod_i \tau(e_i)^5 \frac{B(\log B)^{2^{21}}}{e_1^2e_2^2e_3^2e_4^2T_1} \ll \frac{B(\log B)^{2^{21}}}{T_1},
    \end{equation*}
    which is again satisfactory.
\end{proof}

\subsection{Lattice point counting}

We are now ready for the main counting argument. Due to our preparations, this ends up being relatively simple, iteratively replacing sums by integrals. In particular, we can satisfactorily bound the error terms arising in these approximations using the truncations from Section~\ref{sec:restrictions}, and sum them trivially over the Möbius variables using the truncation introduced in Section~\ref{sec:large_moebius}.

\begin{prop}\label{prop:lattice_point_counting}
    For $\ab' \in \ZZnz^4$ with \eqref{eq:ai_bound_TB} and $\db,\eb \in \ZZ^4_{>0}$ with \eqref{eq:db}, \eqref{eq:eb}, we have
    \begin{equation*}
        |\Gs(\ab',\db,\eb) \cap S^{(W)}(\ab',B)| = \frac{\vol S^{(W)}(\ab',B)}{|a_1a_4|b_{12}b_{23}b_{34}} + O\left(\frac{W^2B}{T_2^{1/3}|a_1 a_2 a_3 a_4|}\right).
    \end{equation*}
\end{prop}

\begin{proof}
    By Lemma~\ref{lem:lattice}, $|\Gs(\ab',\db,\eb) \cap S^{(W)}(\ab',B)|$ can be written as
    \begin{equation*}
        \sums{0<|a_{12}| \le W|B_{12}|\\b_{12} \mid a_{12}}
        \sums{0<|a_{23}|\le W|B_{23}|\\\congr{a_{23}}{\gamma_{23}}{a_4b_{23}}}
        \sums{0<|a_{34}|\le W|B_{34}|\\\congr{a_{34}}{\gamma_{34}}{a_1b_{34}}} \oneb_{S^{(W)}(\ab',B)}.
    \end{equation*}

    By \eqref{eq:sum_congruence}, the inner sum is
    \begin{equation*}
        \int_{(a_{12},a_{23},a_{34}) \in S^{(W)}} \frac{\ddd a_{34}}{|a_1|b_{34}}  + O(1),
    \end{equation*}
    where the error term depends on the number of components of the domain of integration, which is bounded uniformly by \cite[Lemma~3.6]{DF14}.

    The sum of the error term over $a_{12},a_{23}$ is
    \begin{align*}
         &\ll\frac{W|B_{12}|}{b_{12}} \left(\frac{W|B_{23}|}{|a_4|b_{23}} + O(1)\right) \ll \frac{W^2|B_{12}B_{23}|}{|a_4|} = \frac{|a_1|}{|B_{34}|}\cdot \frac{W^2B}{|a_1a_2a_3a_4|}\\
         &\ll \frac{W^2B}{T_2^{1/3}|a_1a_2a_3a_4|},
    \end{align*}
    using $b_{ij} \ge 1$ and \eqref{eq:ai_bound_WB_2} to drop the $O(1)$ in the first step, and \eqref{eq:ai_bound_TB_2} in the final step.

    For the summation of the main term over $a_{23}$, we combine \eqref{eq:sum_congruence} with partial summation as in \cite[Lemma~3.1]{D09} (where the number of monotonous pieces is uniformly bounded by \cite[Lemma~3.6]{DF14}) to obtain
    \begin{equation*}
        \int_{(a_{12},a_{23},a_{34}) \in S^{(W)}} \frac{\ddd a_{34} \ddd a_{23}}{|a_1a_4|b_{23}b_{34}}  + O\left(\sup_{|a_{23}| \le W|B_{23}|} \int_{(a_{12},a_{23},a_{34}) \in S^{(W)}} \frac{\ddd a_{34}}{|a_1|b_{34}}\right).
    \end{equation*}
    By \eqref{eq:restriction_aij}, the integral over $a_{34}$ in the error term is $\ll W|B_{34}|/(|a_1|b_{34})$ independently of $a_{23}$. Hence (using \eqref{eq:ai_bound_TB_2} again) the sum of this error term over $a_{23}$ is
    \begin{equation*}
        \ll \frac{W|B_{12}|}{b_{12}} \cdot \frac{W|B_{34}|}{|a_1|b_{34}} \ll \frac{W^2|B_{12}B_{34}|}{|a_1|} = \frac{|a_4|}{|B_{23}|}\cdot \frac{W^2B}{|a_1a_2a_3a_4|} \ll \frac{W^2B}{T_2^{1/3}|a_1a_2a_3a_4|}.
    \end{equation*}

    Similarly, summing the main term over $a_{12}$ leads to
    \begin{equation*}
        \int_{S^{(W)}} \frac{\ddd a_{34} \ddd a_{23} \ddd a_{12}}{|a_1a_4|b_{12}b_{23}b_{34}}  + O\left(\sup_{|a_{12}| \le W|B_{12}|} \int_{(a_{12},a_{23},a_{34}) \in S^{(W)}} \frac{\ddd a_{34}\ddd a_{23}}{|a_1a_4|b_{23}b_{34}}\right),
    \end{equation*}
    where the error term (with the integral independent of $a_{12}$) is
    \begin{equation*}
        \ll \frac{W^2|B_{23}B_{34}|}{|a_1a_4|b_{23}b_{34}} \ll \frac{W^2B}{T_2^{1/3}|a_1a_2a_3a_4|}.\qedhere
    \end{equation*}
\end{proof}

Let
\begin{equation}\label{eq:def_theta}
    \theta(\ab',T_1):= \sums{\db: \eqref{eq:db},\ |d_i| \le T_1\\\eb: \eqref{eq:eb},\ |e_i| \le T_1} 
        \frac{\mu(\db,\eb)}{b_{12}b_{23}b_{34}}.
\end{equation}

\begin{prop}\label{prop:summary}
    We have
    \begin{equation*}
        |\Mover(B)| = \sums{\ab' \in \ZZnz^4\\\eqref{eq:ai_bound_TB}} \frac{\theta_0(\ab')\theta(\ab',T_1) \vol S^{(W)}(\ab',B)}{|a_1a_4|} + O(B(\log B)^3\log \log B).
    \end{equation*}
\end{prop}

\begin{proof}
    Plugging Proposition~\ref{prop:lattice_point_counting} into Proposition~\ref{prop:restrict_moebius} gives the main term as stated and the additional error term (using $|\theta_0(\ab')| \le 1$ and $|\mu(\db,\eb)|\le 1$)
    \begin{equation*}
        \ll \sums{\ab' \in \ZZnz^4\\\eqref{eq:ai_bound_TB}} \sums{\db: \eqref{eq:db},\ |d_i| \le T_1\\\eb: \eqref{eq:eb},\ |e_i| \le T_1} 
        \frac{W^2B}{T_2^{1/3}|a_1\dots a_4|} \ll \frac{T_1^8W^2 B(\log B)^4}{T_2^{1/3}},
    \end{equation*}
    which is not larger than the previous error term by our choice of $T_1,T_2,W$.
\end{proof}

\subsection{The local densities}

The next step is to compute the local densities. Here, we first need to remove the truncation for $a_{ij}$ arising from Proposition~\ref{prop:restrict_W} for the real density and the one on the Möbius variables arising from Proposition~\ref{prop:restrict_moebius} for the $p$-adic densities.

\begin{lemma}\label{lem:real_density}
    For $\ab' \in \ZZnz^4$, we have
    \begin{equation*}
        \vol S^{(W)}(\ab',B) = (1+ O(W^{-3}))\cdot\frac{2\omega_\infty(X)}{3}\cdot\frac{B}{|a_2a_3|},
    \end{equation*}
    with $\omega_\infty(X)=2\pi^2$.
\end{lemma}

\begin{proof}
    For the first statement, the change of our three coordinates $a_{ij} = B_{ij}z_{ij}$ appearing in $S^{(W)}$ has Jacobian determinant
    \begin{equation*}
        |B_{12}B_{23}B_{34}| = \frac{B}{|a_2a_3|}.
    \end{equation*}
    For the three dependent variables $a_{ij}$ as in \eqref{eq:dependent_aij}, we observe that
    \begin{equation*}
        a_{13} = \frac{a_2B_{23}z_{23}-a_4B_{34}z_{34}}{a_1} = B_{13}(z_{23}-z_{34})
    \end{equation*}
    and analogously
    \begin{equation*}
        a_{24} = B_{24}(z_{23}-z_{12}),\quad a_{14} = B_{14}(z_{23}-z_{34}-z_{12}).
    \end{equation*}
    Therefore, writing $(z_{13},z_{24},z_{14}):=(z_{23}-z_{34},z_{23}-z_{12},z_{23}-z_{34}-z_{12})$, we have $a_{ij}=B_{ij}z_{ij}$ for all six coordinates (turning \eqref{eq:restriction_aij} into $|z_{ij}|\le W$). Hence the anticanonical monomials appearing in our height condition \eqref{eq:height_on_torsor} transform as
    \begin{equation*}
        |a_{ij}a_ja_{jk}a_ka_{kl}| = |a_ja_kB_{ij}B_{jk}B_{kl}z_{ij}z_{jk}z_{kl}| = B|z_{ij}z_{jk}z_{kl}|
    \end{equation*}
    so that \eqref{eq:height_on_torsor} is equivalent to $\max|z_{ij}z_{jk}z_{kl}| \le 1$.
    Hence
    \begin{equation*}
        \vol S^{(W)}(\ab',B) = \frac{B}{|a_2a_3|} \vol\{(z_{12},z_{23},z_{34}) \in \RR^3 : \max|z_{ij}z_{jk}z_{kl}| \le 1,\ |z_{ij}|\le W\}.
    \end{equation*}

    We claim that the volume appearing on the right hand side is
    \begin{equation}\label{eq:volume_with_error}
         \vol\{(z_{12},z_{23},z_{34}) \in \RR^3 : \max|z_{ij}z_{jk}z_{kl}| \le 1\} + O(W^{-3}).
    \end{equation}
    In other words, we need to bound the contribution of $z_{ij}$ where $\max |z_{ij}|>W$. As in the proof of Proposition \ref{prop:restrict_W}, it suffices to consider the case where $|z_{12}|, |z_{23}| \ll 1/W^2$ and $z_{34} \ll W$. But this volume is clearly $O(W^{-3})$.

    Finally, the change of variables $z_{12}=y_3$, $z_{23}=y_1$, $z_{34}=y_1-y_2$ shows that the volume in \eqref{eq:volume_with_error} is $2\omega_\infty(X)/3$. For the computation of the volume in the definition of $\omega_\infty(X)$, we refer to \cite[\S 1.3]{Bre02}.
\end{proof}

\begin{lemma}\label{lem:p-adic_densities}
    For $\ab' \in \ZZnz^4$ with $\theta_0(\ab')=1$, we have
    \begin{equation*}
        \theta(\ab',T_1) = \theta(\ab')+O(T_1^{-1/2}\tau(a_1a_2a_3a_4)),
    \end{equation*}
    where
    \begin{equation*}
        \theta(\ab') = \prod_{p \mid a_1a_2a_3a_4} \left(1-\frac 1 p\right)\left(1-\frac 1{p^2}\right)\prod_{p\nmid a_1a_2a_3a_4} \left(1-\frac 4{p^2}+\frac{3}{p^3}\right).
    \end{equation*}
    Here, the error term contributes $O(T_1^{-1/2}B(\log B)^8)$ to $|\Mover(B)|$, which is satisfactory.
\end{lemma}

\begin{proof}
    By its definition \eqref{eq:def_theta}, we can write
    \begin{align*} 
    \theta(\ab',T_1)&=\sums{\db: \eqref{eq:db},\ |d_i| \le T_1\\\eb: \eqref{eq:eb},\ |e_i| \le T_1} \frac{\mu(\db,\eb)}{e_1e_2e_3e_4[d_1;d_2;(d_3;d_4)][d_2;d_3;d_4][d_1;d_3;d_4]}\\
    &=D(\ab',T_1) \prod_{i=1}^4 E(a_i,T_1)
    \end{align*}
    with
    \begin{equation*} D(\ab',T_1):=\sum_{\db: \eqref{eq:db},\ |d_i| \le T_1} \frac{\mu(\db)}{[d_1;d_2][d_2;d_3;d_4][d_1;d_3;d_4]}
    \end{equation*}
    and
    \begin{equation*}
        E(a,T_1):=\sum_{e \mid a,\ 1 \le e \le T} \frac{\mu(e)}{e}.
    \end{equation*}
    Writing 
    \begin{equation*}
        E(a):=\sum_{e \mid a} \frac{\mu(e)}{e}=\prod_{p \mid a} \left(1-\frac{1}{p}\right),
    \end{equation*}
    we clearly have $E(a) \in (0,1]$ and 
    \begin{equation*}
        |E(a)-E(a,T_1)| \le \sum_{e \mid a,\ e>T_1} \frac{1}{e} < \frac{\tau(a)}{T_1}.
    \end{equation*}
    Moreover, we note that
    \begin{align*}
        D(\ab')&:=\sum_{\db: \eqref{eq:db}} \frac{\mu(\db)}{[d_1;d_2][d_2;d_3;d_4][d_1;d_3;d_4]}\\&=\prod_{p \mid a_1a_2a_3a_4} \left(1-\frac{1}{p^2}\right) \prod_{p \nmid a_1a_2a_3a_4} \left(1-\frac{4}{p^2}+\frac{3}{p^3}\right),
    \end{align*}
    where we obtain the Euler product using multiplicativity and checking which $d_i$ can be divisible by $p$ (depending on whether $p$ divides one of the $a_i$ or not).

    Finally, using symmetry and then applying Rankin's trick, we have
    \begin{align*}
        |D(\ab')-D(\ab',T_1)| &\ll \sum_{\db: d_1>T_1} \frac{\mu(\db)^2}{[d_1;d_2][d_2;d_3;d_4][d_1;d_3;d_4]}\\
        &\le \frac{1}{T_1^c}\sum_{\db} \frac{\mu(\db)^2 d_1^c}{[d_1;d_2][d_2;d_3;d_4][d_1;d_3;d_4]}\\
        &\ll \frac{1}{T_1^c} \prod_p \left(1+O(p^{c-2})\right) \ll \frac{1}{T_1^c}
    \end{align*}
    for any fixed $c \in (0,1)$ by computing the Euler factors as above. Collecting the results, we see that all error terms are bounded by $\tau(a_1a_2a_3a_4)/T_1^{1/2}$ as desired.

    Finally, the total contribution of the error term can be estimated as
    \begin{equation*}
        \ll \frac{1}{T_2^{1/2}}\sums{\ab' \in \ZZnz^4\\\eqref{eq:ai_bound_TB}} \frac{B\tau(a_1)\tau(a_2)\tau(a_3)\tau(a_4)}{|a_1a_2a_3a_4|} \ll \frac{B(\log B)^8}{T_2^{1/2}}.\qedhere
    \end{equation*}
\end{proof}

\subsection{Completion of the proof}

Finally, we perform the summation over the remaining four variables $\ab'$. Here, we are summing a nicely controlled arithmetic function against a non-negative weight function. In principle, this lies in the scope of partial summation, although we would need to take care with powers of $\log B$ arising in the error terms. This situation however is well-understood, and we can use the machinery of \cite[\S 4, \S 7]{D09}; here, the simplified version of \cite[\S 6.4.4]{ADHL} is sufficient.

\begin{lemma}\label{lem:sum_ai_alpha}
    We have
    \begin{equation*}
        \sums{\ab' \in \ZZnz^4\\\eqref{eq:ai_bound_TB}} \frac{\theta_0(\ab')\theta(\ab')B}{|a_1a_2a_3a_4|} = 2^4\cdot\frac{ 3\alpha(X)}{5} \theta_1 B(\log B)^4 + O(B(\log B)^3 \log \log B),
    \end{equation*}
    where
    \begin{equation*}
        \theta_1 = \prod_p \left(1-\frac 1 p\right)^5\left(1+\frac 5 p+\frac 1{p^2}\right).
    \end{equation*}
\end{lemma}

\begin{proof}
    Restricting the summation over $\ab'$ to $\ZZ_{>0}^4$ gives a factor $2^4$ by symmetry. We apply \cite[Proposition~6.4.4.7]{ADHL} to the volume function that is the product of $B/(a_1a_2a_3a_4)$ and the indicator function of all $a_1,\dots,a_4 \ge 1$ satisfying \eqref{eq:ai_bound_TB}, and the arithmetic function $\theta_0(\ab')\theta(\ab') \in \Theta_4'(4)$ as in \cite[Definition~6.4.4.4]{ADHL}. The result is
    \begin{equation*}
        2^4 \theta_1 V_0(B)+O(B(\log B)^3 \log \log B),
    \end{equation*}
    where the arithmetic factor $\theta_1$ as above is obtained by a repeated application of \cite[Proposition~6.4.4.5]{ADHL}, and
    \begin{equation*}
        V_0(B):=\int_{\substack{a_1,a_2,a_3,a_4\ge 1\\\eqref{eq:ai_bound_TB}}} \frac{B}{a_1a_2a_3a_4} \ddd a_1 \ddd a_2 \ddd a_3 \ddd a_4.
    \end{equation*}
    Substituting $a_i = (B/T_2)^{x_i}$ and using $\log T_2 \ll \log \log B$, we obtain
    \begin{equation*}
        V_0(B) = V_1 \cdot B(\log(B/T_2))^4 = V_1 B(\log B)^4 + O(B(\log B)^3\log \log B),
    \end{equation*}
    where
    \begin{equation*}
        V_1:=\vol\{(x_1,x_2,x_3,x_4) \in \RR_{\ge 0}^4 : 2x_i+2x_j+2x_k-x_l\le 1\} = \frac{1}{180} = \frac{3\alpha(X)}{5}
    \end{equation*}
    by \cite[(3.25)]{Bre02}.
\end{proof}

\begin{proof}[Proof of the Theorem]
    We plug Lemma~\ref{lem:real_density} and Lemma~\ref{lem:p-adic_densities} into Proposition~\ref{prop:summary} to obtain
    \begin{equation*}
         |\Mover(B)| =  (1+ O(W^{-3}))\cdot\frac{2\omega_\infty(X)}{3} \sums{\ab' \in \ZZnz^4\\\eqref{eq:ai_bound_TB}} \frac{\theta_0(\ab')\theta(\ab')B}{|a_1a_2a_3a_4|} + O(B(\log B)^3\log \log B).
    \end{equation*}
    Applying Lemma~\ref{lem:sum_ai_alpha} gives
    \begin{equation*}
        |\Mover(B)| = \frac{2^5\alpha(X)}{5}\theta_1\omega_\infty(X)B(\log B)^4+O(B(\log B)^3 \log\log B)
    \end{equation*}
    since $O(W^{-3}B(\log B)^4)$ is smaller than our error term.
    
    It is not hard to check that we obtain the same bound for $|\Munder(B)|$, and hence plugging this into Lemma~\ref{lem:symmetry} and Proposition~\ref{prop:parameterization} gives the result.    
\end{proof}

\bibliographystyle{amsalpha}

\bibliography{main}

\end{document}